\documentclass[12pt]{amsart}
\usepackage{amsmath,amssymb,amsthm}
\usepackage{verbatim}
\usepackage{color,xcolor}
\usepackage{enumerate}
\usepackage{mathrsfs}
\usepackage{mathtools}

\usepackage{bm}

\usepackage[colorlinks=true, bookmarks=true,pdfstartview=FitV, linkcolor=blue, citecolor=blue, urlcolor=blue]{hyperref}

\calclayout

\allowdisplaybreaks
\numberwithin{equation}{section}

\newtheorem{theorem}{Theorem}[section]
\newtheorem{lemma}[theorem]{Lemma}

\newtheorem{proposition}[theorem]{Proposition}

\theoremstyle{definition}

\newtheorem{remark}[theorem]{Remark}

\theoremstyle{remark}

\newcommand{\R}{\mathbb R}

\newcommand{\F}{{\mathcal F}}
\newcommand{\Ss}{\mathcal S}
\newcommand{\Oo}{\mathcal O}

\def\Xint#1{\mathchoice 
{\XXint\displaystyle\textstyle{#1}}%
{\XXint\textstyle\scriptstyle{#1}}%
{\XXint\scriptstyle\scriptscriptstyle{#1}}%
{\XXint\scriptscriptstyle\scriptscriptstyle{#1}}%
\!\int} 
\def\XXint#1#2#3{{\setbox0=\hbox{$#1{#2#3}{\int}$} 
\vcenter{\hbox{$#2#3$}}\kern-.5\wd0}}

\def\avgint{\Xint-}

\DeclareMathOperator{\supp}{\mathrm{supp}}
\DeclareMathOperator*{\esssup}{ess\,sup}
\DeclareMathOperator*{\essinf}{ess\,inf}

\newcommand{\pp}{{p(\cdot)}}

\newcommand{\Lp}{L^{p(\cdot)}}

\newcommand{\Pp}{\mathcal P}
\newcommand{\qq}{{q(\cdot)}}
\newcommand{\rr}{{r(\cdot)}}

\newcommand{\M}{\mathcal{M}}

\newcommand{\Rdf}{\mathcal{R}}

\begin{document}

\title[Norm inequalities on Hardy spaces]
{A new approach to norm inequalities on weighted and variable Hardy spaces}

\author[Cruz-Uribe]{David Cruz-Uribe, OFS}
\address{Department of Mathematics, University of Alabama, Tuscaloosa, AL 35487}
\email{dcruzuribe@ua.edu}

\author[Moen]{Kabe Moen}
\address{Department of Mathematics, University of Alabama, Tuscaloosa, AL 35487}
\email{kabe.moen@ua.edu}

\author[Nguyen]{Hanh Van Nguyen}
\address{Department of Mathematics, University of Alabama, Tuscaloosa, AL 35487}
\email{hvnguyen@ua.edu}



\thanks{The first author is supported by research funds from the
  Dean of the College of Arts \& Sciences, the University of Alabama, and the second author is supported by the Simons Foundation.}

\date{February 9, 2019}

\begin{abstract}
We give new proofs of Hardy space estimates for fractional and
singular integral operators on weighted and variable exponent Hardy
spaces.    Our proofs consist of several interlocking ideas:  finite
atomic decompositions in terms of $L^\infty$ atoms, vector-valued
inequalities for maximal and other operators, and Rubio de
Francia extrapolation.  Many of these estimates are not new, but we
give new and substantially simpler proofs, which in turn significantly
simplifies the proofs of the Hardy spaces inequalities.
\end{abstract}

\maketitle

\section{Introduction}
\label{section:intro}

In this paper we give new 
proofs of norm inequalities for Calder\'on-Zygmund singular integrals
and fractional integral operators on the weighted Hardy spaces,
$H^p(w)$, and the variable Hardy spaces, $H^\pp$.  The theory of
weighted Hardy spaces is classical:  see the monograph by Str\"omberg
and Torchinsky~\cite{MR1011673} and the earlier paper by Garc\'\i
a-Cuerva~\cite{MR549091}.  Variable Hardy spaces are Hardy spaces
defined in the scale of the variable Lebesgue spaces $L^\pp$, a
generalization of the $L^p$ spaces that has been an active area of research
for the past two decades:  see the books~\cite{cruz-fiorenza-book,
  diening-harjulehto-hasto-ruzicka2010}.  The variable Hardy spaces
were introduced more recently:  see~\cite{DCU-dw-P2014, MR2899976}.
(Complete definitions of these spaces will be given in
Section~\ref{section:prelim} below.)

We give Hardy space estimates for three types of
operators:  singular integrals of convolution type, fractional
integral operators (which are also convolution operators), and
singular integrals of non-convolution type.  
A Calder\'on-Zygmund singular integral of convolution type
is an operator $T$ such that for all $f\in C_c^\infty$, 
\[ Tf(x) =  \text{p.v.}\int_{\R^n} K(x-y) f(y)\,dy , \]
where the kernel $K$ is defined on $\R^n\setminus \{0\}$ and has
regularity of order $N+1$:
\[  |\partial_x^{\alpha} K(x)|
\leq 
\frac{A_\alpha}{|x|^{n+|\alpha|}} \]
for all $\alpha$ such that
$|\alpha|\leq N+1$, where $N$ is a sufficiently
large integer.    For $0<p\leq 1$, if $N> \lfloor
n\big(\frac{1}{p}-1\big)\rfloor$, then $T : H^p \rightarrow H^p$.
(See Stein~\cite{stein93} or Garc\'ia-Cuerva and Rubio de
Francia~\cite{garcia-cuerva-rubiodefrancia85}.)    

Our first two theorems extend this result to weighted and variable
exponent Hardy spaces; again, for brevity we defer some
technical definitions to Section~\ref{section:prelim}.  For a weight
$w$ let $r_w= \inf\{ p : w\in A_p\}$.

\begin{theorem} \label{thm:sio-weight} Given a weight $w\in A_\infty$
  and $0<p<\infty$, suppose that $T$ is a Calder\'on-Zygmund singular
  integral operator of convolution type with regularity of order
  $N+1$, where
\[ N> \bigg\lfloor
n\bigg(\frac{r_w}{p}-1\bigg)\bigg\rfloor. \]
Then 
$T : H^p(w) \rightarrow H^p(w).$
\end{theorem}

\begin{remark}
Weighted Hardy space estimates for singular integrals (actually for
the more general class of multipliers) were proved by Str\"omberg and
Torchinsky~\cite{MR1011673}.  Theorem~\ref{thm:sio-weight} was proved
by Lu and Zhu~\cite{MR2927673} for singular integrals with $C^\infty$
kernels; see the bibliography of their paper
for earlier results. 
\end{remark}

\begin{theorem} \label{thm:sio-var}
Given an exponent function $\pp\in \Pp_0$, suppose $0<p_-\leq
p_+<\infty$ and $\pp \in LH$.   Suppose further that $T$ is a Calder\'on-Zygmund singular integral operator of convolution
type with regularity of order $N+1$, where
\[ N> \bigg\lfloor
n\bigg(\frac{1}{p_-}-1\bigg)\bigg\rfloor. \]
Then
$ T : H^\pp\rightarrow H^\pp.$
\end{theorem}

\begin{remark}
Theorem~\ref{thm:sio-var} was first proved independently
in~\cite{DCU-dw-P2014, MR2899976}.  
\end{remark}

\medskip

Second, we consider the fractional integral operator, $I_\alpha$.
Given $0<\alpha<n$, define
\[  I_\alpha f(x) = \int_{\R^n} \frac{f(y)}{|x-y|^{n-\alpha}}\,dy.  \]
If $0<p<\frac{n}{\alpha}$ and
$\frac{1}{p}-\frac{1}{q}=\frac{\alpha}{n}$, then $I_\alpha : H^p
\rightarrow H^q$.  (See Stein~\cite{stein93} or Krantz~\cite{MR667967}.)  

\begin{theorem} \label{thm:fracint:wtd}
Given $0<\alpha<n$, $0<p<\frac{n}{\alpha}$, define $q$ by
$\frac{1}{p}-\frac{1}{q}=\frac{\alpha}{n}$.  If a weight $w$ is such
that $w^p \in RH_{\frac{q}{p}}$, then
$I_\alpha : H^p(w^p) \rightarrow H^q(w^q).$
\end{theorem}

\begin{remark}
Theorem~\ref{thm:fracint:wtd} was first proved by Str\"omberg and
Wheeden~\cite{MR766221}; see also Gatto, Guti\'errez and
Wheeden~\cite{MR784004}.  
\end{remark}

\begin{theorem} \label{thm:fracint-var}
Given $0<\alpha<n$, and $\pp \in \Pp_0$, suppose  
$0<p_-\leq p_+< \frac{n}{\alpha}$ and $\pp \in LH$. Define  $\qq$ by
$\frac{1}{\pp}-\frac{1}{\qq} = \frac{\alpha}{n}$.  Then 
$ I_\alpha : H^\pp \rightarrow H^\qq.$
\end{theorem}

\begin{remark}
Theorem~\ref{thm:fracint-var} was first proved by Rocha and
Urciuolo~\cite{MR3233547}.  
Sawano~\cite[Theorem~5.1]{MR3090168} states this result, but it seems that the proof only shows for $I_\alpha :
H^\pp \rightarrow L^\qq$.  
\end{remark}

\medskip

Third, we consider Calder\'on-Zygmund operators of non-convolution
type, as defined by Coifman and Meyer~\cite{MR518170}.  An
operator $T$ is a Calder\'on-Zygmund operator if it is a bounded
operator on $L^2(\R^n)$, and if all $f\in L^\infty_c$ and
$x\not\in\supp(f)$,
\[ Tf(x) = \int_{\R^n} K(x,y)f(y)\,dy, \]
where the  distributional kernel coincides with a function $K$ defined
away the diagonal on $\mathbb{R}^{2n}$ and satisfies the standard
 estimates
\begin{gather}
\label{eq.8001}
|K(x,y)|\le \frac{C}{|x-y|^n},\quad x\ne y,\\
\label{eq.8002}
|K(x,y+h) - K(x,y)|+|K(x+h,y) - K(x,y)|\le \frac{C|h|^\delta}{|x-y|^{n+\delta}}
\end{gather}
for all $|h|\le \frac12|x-y|$, where $C>0$ and $0<\delta\le 1$.

For boundedness on Hardy spaces we will also need to assume two
additional conditions on $T$.  First, we will need that the
kernel $K$ satisfies for some $N>0$ the additional smoothness condition
\begin{equation}
\label{eq.8003}
|\partial^{\beta}_yK(x,y+h)-\partial^{\beta}_yK(x,y)|\le \frac{C|h|^\delta}{|x-y|^{n+N+\delta}}
\end{equation}
for all $x\ne y$, $|h|\le \frac{|x-y|}2$ and all $|\beta|=N$. 
Second, we will need that the operator $T$ has $L$ vanishing moments,
in the sense that 
\begin{equation}
\label{eq.8004}
\int x^{\beta}Ta(x)dx=0
\end{equation}
for all $(L+1,\infty)$ atoms $a$ and $|\beta|\le L$.  We note that
this moment condition is satisfied by all convolution type singular
integrals (see~\cite[Lemma~2.1]{GNNS-preprint}), and is a necessary
condition for $T$ to map into unweighted $H^p$
(see~\cite[Theorem~7]{MR3525560}).

\begin{theorem} \label{thm:wtd-nonconv}
  Given $w\in A_\infty$ and $0<p<\infty$, suppose that $T$ is a
  Calder\'on-Zygmund operator associated with a kernel $K$ that
 satisfies \eqref{eq.8003} for all $|\beta| = L+1$, and suppose $T$ has $L$
 vanishing moments \eqref{eq.8004}, where
\[
L = \max\bigg(\bigg\lfloor
n\bigg(\frac{r_w}{p}-1\bigg)\bigg\rfloor,-1\bigg).
\]
(If $L=-1$, then condition \eqref{eq.8004} can be omitted.)  Then
$T: H^p(w) \rightarrow H^p(w)$.
\end{theorem}

\begin{remark}
Norm inequalities for non-convolution Calder\'on-Zygmund operators
have been considered by a number of authors.  In the unweighted case,
Alvarez and Milman~\cite{MR3785798} show that if $T$ is a
Calder\'on-Zygmund operator, then $T : H^p \rightarrow H^p$ for
$\frac{n}{n+\delta}<p\leq 1$, where $\delta$ is the exponent in
\eqref{eq.8003}.   From Theorem~\ref{thm:wtd-nonconv} we get the
slightly larger range $\frac{n}{n+1}<p\leq 1$ but only if we assume
that $L=0$, so that we require greater regularity on the operator $T$
than they do.

Their results were generalized to the full range of $0<p\leq 1$ in the
unweighted case by Hart and Lu~\cite{MR3448918},  and to $0<p<\infty$
and $w\in A_\infty$ in
the weighted case by Hart and Oliveira~\cite{MR3649238}.  Their
results are not directly comparable to ours, though the conditions on
the weights and the range of $p$ are roughly the same.   They assume a great
deal more regularity on the kernel $K$:  they require that a version
of~\eqref{eq.8003} holds that involves derivatives in both $x$ and
$y$, whereas we only require derivatives in $y$.   We also note that
rather than \eqref{eq.8004}, they require that $T^*(x^\beta)=0$.  The
two are formally equivalent, but the latter requires additional machinery to define.  
\end{remark}

\begin{theorem} \label{thm:var-nonconv}
  Given an exponent function $\pp \in \Pp_0$, suppose $0<p_-\leq
  p_+<\infty$ and $\pp \in LH$.  Suppose further that $T$ is a
  Calder\'on-Zygmund operator  associated with a kernel $K$ that
 satisfies \eqref{eq.8003} for all $|\beta| = L+1$, and suppose $T$ has $L$
 vanishing moments \eqref{eq.8004}, where
\[
L = \max\bigg(\bigg\lfloor
n\bigg(\frac{1}{p_-}-1\bigg)\bigg\rfloor,-1\bigg).
\]
(If $L=-1$, then condition \eqref{eq.8004} can be omitted.)  Then
$T: H^\pp \rightarrow H^\pp$.
\end{theorem}

\begin{remark}
  Theorem~\ref{thm:var-nonconv} is new.  However, a slightly weaker
  result was implicitly proved as a special case of a result for multilinear
  Calder\'on-Zygmund operators recently proved in~\cite{1708.07195}.
\end{remark}

\medskip

Most of our results are not new; however,  our main contribution in
this paper is our new approach to the proofs, which we believe is
significantly simpler and more transparent than existing proofs.
Therefore, before giving the actual proofs, we want
to summarize their main ideas.

The first component is a
finite atomic decomposition in terms of $L^\infty$ atoms.  In the
unweighted case such a decomposition was first proved by Meda,
Sj\"ogren and Vallarino~\cite{MR2399059}.  They showed that on a dense
set in $H^p$, $0<p<1$, it is possible to write a function $f$ as a
finite sum of $(N,\infty)$ atoms
\[   f = \sum_{i=1}^M \lambda_i a_i \]
in such a way that
\[ \|f\|_{H^p} \approx \bigg(\sum_{i=1}^M
  \lambda_i^p\bigg)^{\frac{1}{p}}. \]
The advantage of such a decomposition is that it allows the
interchange of the operator and the sum without having to worry about
the convergence of the sum, and reduces the problem to estimating the
operator on individual atoms.  In a previous
paper~\cite{1708.07195} we extended the finite atomic decomposition to weighted Hardy
spaces; here we prove it for variable Hardy spaces, generalizing a
result proved in~\cite{DCU-dw-P2014}.

The second component of our proofs starts with a vector-valued inequality
due to Grafakos and Kalton~\cite{MR1852036}:  given any collection of
cubes $Q_k$ and functions $g_k$ with $\supp(g_k)\subset Q_k$, then for
$0<p\leq 1$,
\[ \bigg\|\sum_k g_k \bigg\|_p 
\lesssim 
\bigg\|\sum_k \bigg(\avgint_{Q_k} g_k\,dx\bigg)\chi_{Q_k}
\bigg\|_p. \]
Their proof was quite technical, since the cubes are not assumed to be
disjoint.  Though not needed, we give a very elementary proof of this
inequality of Grafakos and Kalton.  More importantly, however, we prove versions that
hold for $p>1$ on weighted spaces, and on variable exponent spaces.
Using these, we can divide the estimate of the operator on individual
atoms into their local and global pieces and then estimate the local
piece using unweighted $L^q$ estimates for $q>\max(p,1)$.  We also give a new
proof of a variant of the Grafakos-Kalton lemma that is used in the
off-diagonal case for the fractional integral operator and that in the
weighted case is due to Str\"omberg and Wheeden~\cite{MR766221} and in
the variable exponent case to Sawano~\cite{MR3090168}.  Again our
proofs are much simpler than the original ones.

The third component of our proofs are vector-valued inequalities for
the Hardy-Littlewood maximal operator.  It is a classical result due to Fefferman and
Stein~\cite{fefferman-stein71} that for $1<p,r<\infty$,
\[ \bigg\|\bigg( \sum_k (Mg_k)^r\bigg)^{\frac{1}{r}} \bigg\|_p 
\lesssim
\bigg\|\bigg( \sum_k |g_k|^r\bigg)^{\frac{1}{r}} \bigg\|_p. \]
A similar inequality also holds for the fractional maximal operator,
$M_\alpha$, $0<\alpha<n$, and is due to Ruiz and
Torrea~\cite{MR831197}: if $1<p<\frac{n}{\alpha}$,
$\frac{1}{p}-\frac{1}{q}=\frac{\alpha}{n}$, and $1<r<\infty$,
\[ \bigg\|\bigg( \sum_k (M_\alpha g_k)^r\bigg)^{\frac{1}{r}} \bigg\|_q 
\lesssim
\bigg\|\bigg( \sum_k |g_k|^r\bigg)^{\frac{1}{r}} \bigg\|_p. \]
Both of these inequalities extend to weighted and variable Lebesgue
spaces, and we use them to estimate the global part of the operator applied to an
atom.  

The final component of our proofs is the theory of Rubio de Francia
extrapolation.  Using the reformulation in terms of extrapolation
pairs~\cite{MR2797562}, the above vector-valued inequalities for
maximal operators are an immediate consequence of the corresponding
scalar inequalities.  The variants of the Grafakos-Kalton lemma
described above also follow easily from extrapolation.  For these we
need to use more recent versions of extrapolation, including limited
range extrapolation~\cite{MR3788859}, extrapolation with respect to
reverse H\"older weights~\cite{MR3785798}, and extrapolation into
variable Lebesgue
spaces~\cite{cruz-fiorenza-book,cruz-uribe-fiorenza-martell-perez06,
  MR2797562}.  Our proofs also depend on several new extrapolation
results which we prove here.

\medskip

The remainder of this paper is organized as follows.  In
Section~\ref{section:prelim} we give the necessary definitions and
results about weighted and variable exponent Hardy spaces.  In
Section~\ref{section:extrapol} we give the versions of extrapolation
that we use and prove the new versions we need.  In
Section~\ref{section:lemmas} we state and prove the vector-valued
inequalities we use.  In Section~\ref{section:sio} we prove
Theorems~\ref{thm:sio-weight} and~\ref{thm:sio-var}; in
Section~\ref{section:frac-int} we prove Theorems~\ref{thm:fracint:wtd}
and~\ref{thm:fracint-var}; and in Section~\ref{section:nonconv} we
prove Theorems~\ref{thm:wtd-nonconv} and~\ref{thm:var-nonconv}.  Even
though the proofs are given in separate sections, we want to emphasize
that they all have a common structure.  
Finally, in Section~\ref{section:BFS} we briefly discuss generalizing
our results to Hardy spaces defined with respect to other scales of
Banach function spaces.

Throughout this paper, $n$ will denote the dimension of the underlying
space, $\R^n$.  By a cube $Q$ we will always mean a cube whose sides
are parallel to the coordinate axes, and for $\tau>0$, $\tau Q$ will
denote the cube with the same center such that
$\ell(\tau Q)=\tau \ell(Q)$.  Given $Q$, let $Q^*=2\sqrt{n}Q$ and
$Q^{**}=(Q^*)^*$.  We define the average of a function $f$ on $Q$ by
$\avgint_Q f\,dx = |Q|^{-1}\int_Q f\,dx$.  By $C,\,c$, etc. we will
mean constants that may depend on the underlying parameters in the
proof.  Their values may change at each appearance.  Given two
quantities $A$ and $B$, we will write $A\lesssim B$ if there exists a
constant $c$ such that $A\leq cB$.  If $A\lesssim B$ and
$B\lesssim A$, we write $A\approx B$.

\section{Preliminaries}
\label{section:prelim}

In this section we gather together some basic results about weighted
and variable exponent spaces.  We begin with some information about
weights.  For more information, see~\cite{MR2797562, duoandikoetxea01,
  garcia-cuerva-rubiodefrancia85}.   By a weight we mean a
non-negative, locally integrable function $w$ such that
$0<w(x)<\infty$ a.e.  For $1<p<\infty$, a weight is in the Muckenhoupt
class $A_p$ if for every cube $Q$
\[  \avgint_Q w\,dx \left(\avgint_Q w^{1-p'}\,dx\right)^{p-1} \leq C, \]
and when $p=1$, $w\in A_1$ if for every cube $Q$ and a.e. $x\in Q$,
\[  \avgint_Q w\,dx \leq Cw(x). \]
If $1<p<\infty$ and $w\in A_p$, then the Hardy-Littlewood maximal
operator, defined by
\[ Mf(x) = \sup_Q \avgint_Q |f(y)|\,dy \cdot \chi_Q(x), \]
is bounded on $L^p(w)$.

Define the set 
\[ A_\infty = \bigcup_{p\geq 1} A_p.  \]
Given a weight $w\in A_\infty$, define
\[ r_w = \inf\{ r\geq 1 : w \in A_r \}. \]
A weight $w\in A_\infty$ if and only if $w\in RH_s$ for some $s>1$:
that is, for every cube $Q$,
\[ \left(\avgint_Q w^s\,dx\right)^{\frac{1}{s}}\leq C\avgint_Q
  w\,dx.  \]
Furthermore, we have the property that $w\in RH_s$ if and only if
$w^s\in A_\infty$.   The limiting class $RH_\infty$ is defined to be
all $w$ such that for every cube $Q$ and a.e. $x\in Q$,
\[ w(x) \leq C\avgint_Q w\,dx. \]

Given $1<p,q<\infty$, a weight satisfies the $A_{p,q}$ condition of Muckenhoupt and Wheeden if for every cube $Q$,
\[  \left(\avgint_Q w^q\,dx\right)^{\frac{1}{q}}
 \left(\avgint_Q w^{-p'}\,dx\right)^{
\frac{1}{p'}} \leq C. \]
It follows from the definition that $w\in A_{p,q}$ if and only if $w^q
\in A_{1+\frac{q}{p'}}$.  
When $p=1$ and $q>1$; we say that $w\in
A_{1,q}$ if for every cube $Q$ and almost every $x\in Q$,
\[ \avgint_Q w^q\,dx \leq Cw(x)^q. \]
This is clearly equivalent to $w^q\in A_1$. 
Given $0\leq \alpha<n$ and $1<p<\frac{n}{\alpha}$, define $q$ by
$\frac{1}{p}-\frac{1}{q}=\frac{\alpha}{n}$.  If $w\in A_{p,q}$ then the fractional maximal operator,
\[ M_\alpha f(x) = 
\sup_Q |Q|^{\frac{\alpha}{n}}\avgint_Q |f(y)|\,dy \cdot \chi_Q(x), \]
is bounded from $L^p(w^p)$ to $L^q(w^q)$.  

\medskip

We now define the weighted Hardy spaces $H^p(w)$, $0<p<\infty$.  For
more information, see~\cite{MR1011673}.  Let $\Ss$ denote the Schwartz
class of smooth functions.  Given a (large) integer $N_0$, define
\[ \Ss_{N_0} = \bigg\{ \phi \in \Ss :
\int_{\R^n} (1+|x|)^{N_0} \bigg(\sum_{|\beta|\leq N_0}
|\partial^\beta \phi(x)|^2\bigg)\,dx \leq 1 \bigg\}. \]
Given $\phi \in \Ss_{N_0}$, define the radial maximal operator
\[ M_\phi f(x) = \sup_{t>0} |\phi_t * f(x)|, \]
where $\phi_t(x) = t^{-n}\phi(x/t)$.  Define the grand maximal
operator
\[ \M_{N_0}f(x) = \sup_{\phi\in \Ss_{N_0}} M_\phi f(x). \]
Given $w\in A_\infty$ and $0<p<\infty$, define the weighted Hardy
space to be the set of distributions
\[ H^p(w) = \{ f \in \Ss' : \M_{N_0}f \in L^p(w) \}  \]
with quasi-norm $\|f\|_{H^p(w)}=\|\M_{N_0}f \|_{L^p(w)}$. 
If $p>1$ and $w\in A_p$, then $H^p(w)=L^p(w)$,  since $M_{N_0}f$ is
dominated pointwise by $Mf$, where $M$ is the Hardy-Littlewood maximal
operator.  We can choose the value $N_0$,  depending only on $n$, $p$
and $r_w$, so that $f\in H^p(w)$ if and only if for any $\phi\in \Ss_{N_0}$, $M_\phi
f \in L^p(w)$.     The value of $N_0$ chosen will be implicit in our
constants below.

Given an integer $N>0$, we define an $(N,\infty)$ atom be a bounded
function $a$ such that $\|a\|_\infty \leq 1$,   $\supp(a)\subset Q$
for some cube $Q$, and 
such that for all $|\beta|\leq N$, 
\[ \int_{\R^n} x^\beta a(x)\,dx = 0.  \]
Given $0<p<\infty$ and $w\in A_\infty$, let
\[ s_w = \bigg\lfloor N\bigg(\frac{r_w}{p}-1\bigg)\bigg\rfloor_+.  \]
If $N>s_w$, then any $(N,\infty)$ atom is in $H^p(w)$.  Moreover,
every element of $f$ can be decomposed as the sum of atoms:
if $f \in H^p(w)$, then there exists a sequence of $(N,\infty)$ atoms $\{a_i\}$
with supports $\{Q_i\}$, and $\lambda_i>0$ such that 
\[ f = \sum_i \lambda_i a_i \]
and 
\[ \|f\|_{H^p(w)} \approx \bigg\| \sum_i \lambda_i
  \chi_{Q_i}\bigg\|_{L^p(w)}.  \]

Given $N>s_w$, define
\[ \Oo_N = \{ f\in C_0^\infty : \int_{\R^n} x^\beta f(x)\,dx = 0, 0\leq
  |\beta|\leq N\}. \]
Each element of $\Oo_N$ is a multiple of an $(N,\infty)$ atom, so it
follows from the atomic decomposition that $\Oo_N$ is dense in
$H^p(w)$.  Moreover, we have the following finite atomic decomposition
that was proved in~\cite{1708.07195}.

\begin{proposition} \label{prop:finite-atomic-wts}
Given $0<p<\infty$ and $w\in A_\infty$, fix $N>s_w$.  Then if $f\in
\Oo_N$, there exists a finite sequence $\{a_i\}_{i=1}^M$  of
$(N,\infty)$ atoms with supports $Q_i$,  and a non-negative sequence $\{\lambda_i\}_{i=1}^M$
such that $f = \sum_i \lambda_i a_i$ and
\[ \bigg\| \sum_{i=1}^M \lambda_i \chi_{Q_i}\bigg\|_{L^p(w)} \leq
  C\|f\|_{H^p(w)}. \]  
\end{proposition}

\begin{remark} 
As an immediate consequence of
Proposition~\ref{prop:finite-atomic-wts} and the density of $\Oo_N$ in
$H^p(w)$, to prove that an operator $T$ extends to a bounded operator from $H^p(w)$ to
itself, it suffices to show that if $f$ is a finite sum of
$(N,\infty)$ atoms, then
\[ \|Tf\|_{L^p(w)} 
\lesssim 
\bigg\| \sum_{i=1}^M \lambda_i \chi_{Q_i}\bigg\|_{L^p(w)}.  \]
\end{remark}

\medskip

We now define the variable exponent Hardy spaces.  We first state some
basic results about variable Lebesgue spaces.  For more information,
see~\cite{cruz-fiorenza-book}. Let $\Pp_0$ be the
collection of all measurable functions $\pp : \R^n \rightarrow
(0,\infty)$.  Define
\[ p_- = \essinf_{x\in \R^n}  p(x), \qquad p_+= \esssup_{x\in \R^n}
  p(x). \]
An exponent $\pp \in \Pp_0$ is log-H\"older continuous, denoted by
$\pp \in LH$, if 
\[ |p(x)-p(y)| \leq \frac{C_0}{-\log(|x-y|)}, \qquad
  |x-y|<\frac{1}{2}, \]
and if there exists constants $p_\infty$ and $C_\infty$ such that
\[ |p(x)-p_\infty| \leq \frac{C_\infty}{\log(e+|x|)}. \]

We define $L^\pp$ to be the set of all measurable functions
$f$ such that
\[ \|f\|_\pp 
= 
\inf\bigg\{ \lambda > 0 : \int_{\R^n}
\bigg(\frac{|f(x)|}{\lambda}\bigg)^{p(x)} \,dx \leq 1 \bigg\} <
\infty. \]
This defines a quasi-norm and  $L^\pp$ is a quasi-Banach function
space;  if $p_-\geq 1$, it is a Banach
space.   If $\pp \in \Pp_0$, $p_->1$, and $\pp \in LH$, then the
Hardy-Littlewood maximal operator is bounded on $L^\pp$.  If
$1<p_-\leq p_+\leq\frac{n}{\alpha}$ and $\qq$ is defined by
$\frac{1}{\pp}-\frac{1}{\qq}=\frac{\alpha}{n}$, then the fractional
maximal operator $M_\alpha$ maps $L^\pp$ to $L^\qq$. 

Given $\pp\in \Pp_0$, define the variable Hardy space $H^\pp$ to be
the set of all distributions $f$ such that $\M_{N_0} f \in L^\pp$.
Again, we may fix $N_0$ depending only $\pp$ and $n$ such that $f\in
H^\pp$ if and only if $M_\phi f \in L^\pp$, where $\phi \in \Ss$.   We
have the following atomic decomposition (see~\cite{DCU-dw-P2014}):  if 
\[ N > s_\pp = \bigg\lfloor
  n\bigg(\frac{1}{p_-}-1\bigg)\bigg\rfloor_+, \]
then every $f\in H^\pp$ can be written as the sum of $(N,\infty)$
atoms,
\[ f = \sum_i \lambda_i a_i, \]
and 
\[ \|f\|_{H^\pp} 
\approx 
\bigg\| \sum_i \lambda_i \chi_{Q_i}\bigg\|_{L^\pp}.  \]

\begin{remark}
This is a slightly different formulation of the atomic decomposition
than that given in~\cite{DCU-dw-P2014}.  However, the difference lies
in the normalization of the atoms and is not significant for our
purposes.
\end{remark}

Similar to the weighted theory, if we fix $N>s_\pp$, then the atomic
decomposition implies that $\Oo_N$ is dense in $H^\pp$.   Furthermore,
we have a finite atomic decomposition for elements of $\Oo_N$.   The
proof of the following result is nearly the same as the proof of
Proposition~\ref{prop:finite-atomic-wts}; the proof can be adapted to
the variable Lebesgue space setting using the ideas used to prove a
similar finite atomic decomposition
in~\cite[Theorem~7.8]{DCU-dw-P2014}.

\begin{proposition} \label{prop:finite-atomic-var}
Given $\pp\in \Pp_0$, suppose $0<p_-\leq p_+<\infty$ and $\pp \in
LH$.  Fix $0<p_0<p_-$.  Then for any $N$ such that
\[ N > \bigg\lfloor
  N\bigg(\frac{1}{p_0}-1\bigg)\bigg\rfloor_+, \]
given any $f\in \Oo_N$, there exists a finite sequence
$\{a_i\}_{i=1}^M$ of $(N,\infty)$ atoms with supports $Q_i$, and a
non-negative sequence $\{\lambda_i\}_{i=1}^M$ such that
$f = \sum_i \lambda_i a_i$ and
\[ \bigg\|\sum_{i=1}^M \lambda_i \chi_{Q_i} \bigg\|_{L^\pp} 
\lesssim
\|f\|_{H^\pp}. \]
\end{proposition}

\begin{remark}
Unless $N(p_-^{-1}-1)$ is an integer, the lower bound for $N$ in
Proposition~\ref{prop:finite-atomic-var} is the same as $s_\pp$
defined above.
\end{remark}

\section{Extrapolation results}
\label{section:extrapol}

In this section we review the theory of Rubio de Francia extrapolation
which is used to prove weighted norm inequalities in weighted and variable
exponent spaces.  We also prove two new variants that we will use in
Section~\ref{section:lemmas} below.  

We begin by recalling the abstract notion of extrapolation pairs.  For
more information, see~\cite{MR2797562}.   We define a family of
extrapolation pairs to be a family $\F$ of pairs of non-negative,
measurable functions $(f,g)$.  Whenever we write an inequality of the
form
\[ \|f\|_X \lesssim \|g\|_Y, \qquad (f,g) \in \F, \]
where $\|\cdot\|_X$ and $\|\cdot\|_Y$ are norms in a Banach or
quasi-Banach space, we mean that this inequality holds for every pair
$(f,g)$ in $\F$ such that $\|f\|_X<\infty$, and the constant is
independent of the pair $(f,g)$.   (If $X=Y=L^p(w)$ for some $w\in A_q$,
then the constant can only depend on the $A_q$ constant of $w$ and not
on $w$ itself. A similar restriction is assumed for off-diagonal
estimates.)  The assumption that $\|f\|_X<\infty$ is only there for 
technical reasons in the proof; it can often be ignored.  Given any pair $(f,g)$, we can
replace it by the pairs $\big(\min(f,N)\chi_{B(0,N)},g\big)$; then by
Fatou's lemma (which still holds in the setting of quasi-Banach
function spaces) we have that
$\|\min(f,N)\chi_{B(0,N)}\|_X<\infty$ and 
\[ \lim_{N\rightarrow \infty} \|\min(f,N)\chi_{B(0,N)}\|_X =
    \|f\|_X. \]

To apply extrapolation to prove a norm inequality for an operator $T$,
we can define
\[ \F = \{ (|Tf|,|f|) : f\in L^\infty_c \}; \]
alternatively, we can replace $L^\infty_c$ by $C_c^\infty$, $\Ss$, or any other
appropriate dense subset of the spaces in question.   To prove
vector-valued inequalities for the operator $T$, we can use the pairs
\[ \bigg( \bigg(\sum_{k=1}^M |Tf_k|^r\bigg)^{\frac{1}{r}},
\bigg(\sum_{k=1}^M |f_k|^r\bigg)^{\frac{1}{r}}\bigg), \]
where again the functions $f_k$ are taken from some appropriate dense
subspace.  

\medskip

Below, to prove the vector-valued inequalities for the
Hardy-Littlewood maximal
operator and the fractional operator, we will use extrapolation;
however, these results are more widely known and we will refer the
reader to the literature.   Here, we will first state two more recent
versions of extrapolation in the scale of weighted spaces, and then
prove two similar results which yield Lebesgue space inequalities.

Our first result is extrapolation in the scale of reverse H\"older
weights.  It was proved independently by Martell and
Prisuelos~\cite{martell-prisuelos} and in~\cite{MR3785798}.

\begin{theorem} \label{thm:rh-ACM}
Given $0<q_0<\infty$, suppose that for some $p_0$, $0<p_0\leq q_0$ and
all $w_0\in RH_{(\frac{q_0}{p_0})'}$,
\[ \|f\|_{L^{p_0}(w_0)} \lesssim \|g\|_{L^{p_0}(w_0)}, \qquad (f,g)\in
  \F. \]
Then for every $p$, $0<p<q_0$, and $w\in RH_{(\frac{q_0}{p})'}$,
\[ \|f\|_{L^{p}(w)} \lesssim \|g\|_{L^{p}(w)}, \qquad (f,g)\in
  \F. \]
\end{theorem}

The second result is an off-diagonal, limited range extrapolation
theorem recently proved in~\cite{MR3788859}.  (We note in passing that
this result contains essentially every other extrapolation theorem as
a special case; see the discussion in the above paper for details.)

\begin{theorem} \label{thm:offdiagRHex} Suppose
  $0<r<\infty$ and $0<p_0,q_0<\infty$ satisfy $0<p_0\leq r$,
  $\frac1{q_0}-\frac{1}{p_0}+\frac1{r}\geq 0$ and $\mathcal F$ is a
  family of pairs of functions $(f,g)$.  Further suppose that we have
$$\left(\int_{\R^n}(fw)^{q_0}\,dx\right)^{\frac1{q_0}}
\lesssim 
\left(\int_{\R^n} (gw)^{p_0}\,dx\right)^{\frac{1}{p_0}}$$
for all $(f,g)\in \mathcal F$ and $w\in RH_{(\frac{r}{p_0})'}$.
Then for every $0<p<r$ and $q$ such that
$\frac1p-\frac1q=\frac{1}{p_0}-\frac{1}{q_0}$ and every weight $w$
such that $w^p\in RH_{(\frac{r}{p})'}$, 
$$\left(\int_{\R^n}(fw)^{q}\,dx\right)^{\frac1{q}}
\lesssim 
\left(\int_{\R^n} (gw)^{p}\,dx\right)^{\frac{1}{p}}$$
for all $(f,g)\in \mathcal F$. 
\end{theorem}

We now prove two extrapolation theorems into the scale of variable Lebesgue
spaces.  
In both proofs we use some well-known properties of the variable
Lebesgue space norm; see~\cite{cruz-fiorenza-book} for details.   We
also use the properties of the Rubio de Francia iteration algorithm,
which is an important tool in extrapolation theory.  For a detailed
discussion of this operator, see~\cite{MR2797562}.

The first result is a generalization of Theorem~\ref{thm:rh-ACM}.

\begin{theorem} \label{thm:rh-extrapol}
Given $0<p<q$, suppose that for all $w\in RH_{(\frac{q}{p})'}$,
\[ \|f\|_{L^p(w)} \lesssim \|g\|_{L^p(w)}, \qquad (f,g) \in \F. \]
Then for all $\pp \in \Pp_0$ such that $p<p_-\leq p_+<q$ and $\pp \in LH$, 
\[ \|f\|_{L^\pp} \lesssim \|g\|_{L^\pp}, \qquad (f,g) \in \F. \]
\end{theorem}

\begin{proof}
For brevity, let $\tau=\frac{q}{p}$, and   let $\rr =
\frac{1}{\tau'}\left(\frac{\pp}{p}\right)'$.   Since $\pp \in LH$, $\rr
\in LH$.  Further, we claim that $r_->1$, which would imply that
the maximal operator is bounded on $L^\rr$.  To prove this, note that
$r_->1$ is equivalent to $[ (\pp/p)']_- > \tau'$,  which in turn is
equivalent to $[\pp/p]_+' > \tau'$,  or $[\pp/p]_+ <\tau$, which in
turn is the same as our assumption that $p_+ < q$.

 Therefore, given non-negative $h$, we can define the Rubio de Francia iteration algorithm by
\[ \Rdf h = \sum_{k=0}^\infty \frac{M^k h}{2^k\|M\|_{L^\rr}}.  \]
This operator satisfies the following properties:
\begin{enumerate}
\item $h \leq \Rdf h$;

\item $\|\Rdf h \|_\rr \leq 2 \|h\|_\rr$;

\item $\Rdf h \in A_1$ and $[\Rdf h]_{A_1} \leq 2\|M\|_{L^\rr}$;

\item $\Rdf (h^{\tau'})^{\frac{1}{\tau'}} \in A_1 \cap RH_{\tau'}$.
  
\end{enumerate}

Now fix $(f,g)\in \F$ such that $\|f\|_\pp < \infty$.  Then by duality, 
\[ \|f\|_\pp^p = \|f^p\|_{\pp/p} 
\lesssim \int_{\R^n} f^p h\,dx, \]
where $h\geq 0$, $h \in L^{(\pp/p)'}$ and $\|h\|_{(\pp/p)'}=1$.  Let $H=
\Rdf(h^{\tau'})^{\frac{1}{\tau'}}$.   By H\"older's inequality in the
scale of variable Lebesgue spaces and the above properties, 
\begin{multline*}
\int f^p H\,dx 
\lesssim 
\|f^p\|_{\pp/p} \|H\|_{(\pp/p)'} \\
= 
\|f\|_\pp^p \|\Rdf (h^\tau)\|_\rr
\lesssim 
\|f\|_\pp^p \|h^\tau\|_\rr
= 
\|f\|_\pp^p < \infty. 
\end{multline*}
Therefore, we can apply our hypothesis and repeat the above estimate
with $g$ in place of $f$ to conclude that 
\begin{multline*}
\|f\|_\pp^p 
 \lesssim 
\int_{\R^n} f^p h\,dx
 \leq 
\int_{\R^n} f^p H\,dx \\
 \lesssim 
\int_{\R^n} g^p H\,dx 
 \lesssim 
\|g^p\|_{\pp/p} \|H\|_{(\pp/p)'} 
 \lesssim 
\|g\|_\pp^p. 
\end{multline*}
This completes the proof.
\end{proof}

Our second result extends a special case of Theorem~\ref{thm:offdiagRHex}
to the variable Lebesgue spaces.

\begin{theorem} \label{thm:offdiagrh-extrapol} Given $0<p<q<\infty$, suppose that for all
  $w^p\in RH_{\frac{q}{p}}$,
$$\|f\|_{L^q(w^q)}\lesssim C\|g\|_{L^p(w^p)}, \qquad (f,g)\in \F.$$
Then for for all $p(\cdot)\in\mathcal P_0$ such that $p<p_-\leq
p_+<\frac{1}{\frac1p-\frac1q}$ and $p(\cdot)\in LH,$ 
$$\|f\|_{L^{q(\cdot)}}\lesssim\|g\|_{L^{p(\cdot)}}, \qquad (f,g)\in
\F,$$
where $q(\cdot)$ is defined by $\frac{1}{\pp}-\frac{1}{\qq} = \frac1p-\frac1q.$
\end{theorem}

\begin{proof} 
  The proof is very similar to the proof of Theorem
  \ref{thm:rh-extrapol}, and we will omit some details.  Define
  $r(\cdot)=\frac{p}{q}\left(\frac{p(\cdot)}{p}\right)'$.  Then
  $r(\cdot)\in LH$ and $r_->1$ is equivalent to
  $\frac{1}{p}-\frac1q<\frac1{p_+}$.   Hence, the maximal operator is bounded on
  $L^\rr$.  Define the Rubio de Francia algorithm
$$ \Rdf h=\sum_{k=0}^\infty \frac{M^k h}{2^k\|M\|_{L^{r(\cdot)}}}.$$
Then $\mathcal R$ satisfies

\begin{enumerate}
\item $h \leq \Rdf h$;

\item $\|\Rdf h \|_\rr \leq 2 \|h\|_\rr$;

\item $\Rdf h \in A_1$ and $[\Rdf h]_{A_1} \leq 2\|M\|_{L^\rr}$;
 
\item $(\Rdf h)^{\frac{p}{q}} \in A_1 \cap RH_{\frac{q}{p}}$. 
\end{enumerate}

Now fix $(f,g)\in \mathcal F$ with $\|f\|_{L^{q(\cdot)}}<\infty$.
Then 
$$\|f\|_{q(\cdot)}^q=\|f^q\|_{q(\cdot)/q}\lesssim \int_{\R^n}
f^qh\,dx,$$
where $h\geq 0$ and $\|h\|_{(q(\cdot)/q)'}=1$.  Let
$H=(\mathcal Rh)^{\frac{1}{q}}$, so that $h\leq H^q$, and $H^p\in RH_{\frac{q}{p}}$. 
Then
\begin{multline*}
\int_{\R^n}f^qh\,dx
\leq 
\int_{\R^n} f^qH^q
\lesssim 
\left(\int_{\R^n}g^pH^p\,dx\right)^{\frac{q}{p}} \\
 \lesssim 
\|g^p\|_{p(\cdot)/p}^{\frac{q}{p}}\|H^p\|_{(p(\cdot)/p)'}^{\frac{q}{p}}
=
\|g\|^q_{p(\cdot)}\|H^p\|_{(p(\cdot)/p)'} 
\end{multline*}
Moreover, since $\mathcal R$ was bounded on $L^{r(\cdot)}$,
$$\|H^p\|_{(p(\cdot)/p)'}
=
\|(\mathcal R h)^{\frac{p}{q}}\|_{(p(\cdot)/p)'}
=
\|\mathcal R h\|_{\frac{p}{q}(p(\cdot)/p)'}^{\frac{p}{q}}
\lesssim 
\|h\|_{\frac{p}{q}(p(\cdot)/p)'}^{\frac{p}{q}}.$$
We now claim that 
$$\frac{p}{q}\left(\frac{p(\cdot)}{p}\right)'=\left(\frac{q(\cdot)}{q}\right)'.$$
To see this, observe that this is equivalent to 
\[ \frac{1}{p\Big(\frac{p(\cdot)}{p}\Big)'}
=
\frac{1}{q\Big(\frac{q(\cdot)}{q}\Big)'}, \]
which is equivalent to our assumption that
\[  \frac{1}{\pp}-\frac{1}{\qq}=\frac{1}{p}-\frac{1}{q}. \]
Therefore, we have that 
$$\|h\|_{\frac{p}{q}(p(\cdot)/p)'}^{\frac{p}{q}}= \|h\|_{(\qq/q)'}^{\frac{p}{q}}\lesssim 1;$$
this completes the proof.
\end{proof}

\section{Vector-valued inequalities}
\label{section:lemmas}

We begin this section by stating four vector-valued inequalities for maximal operators on
weighted and variable Lebesgue spaces.  The first, for the
Hardy-Littlewood maximal operator, was originally
proved by Andersen and John~\cite{andersen-john80}; here we want to
stress that it is an immediate consequence via
extrapolation~\cite[Section~3.8]{MR2797562} of the scalar weighted
norm inequalities for the maximal operator.  

\begin{lemma} \label{lemma:fef-stein-wts}
Given $1<p,\,r<\infty$ and $w\in A_p$, 
\[ \bigg\|\bigg( \sum_k (Mg_k)^r\bigg)^{\frac{1}{r}} \bigg\|_{L^p(w)} 
\lesssim
\bigg\|\bigg( \sum_k |g_k|^r\bigg)^{\frac{1}{r}} \bigg\|_{L^p(w)}. \]
\end{lemma}

The second result on variable Lebesgue spaces, is also an immediate
consequence of the scalar weighted norm
inequalities and extrapolation~\cite[Theorem~4.25]{MR2797562},
\cite[Corollary~5.34]{cruz-fiorenza-book}.

\begin{lemma} \label{prop:fef-stein-var}
Given $\pp \in \Pp_0$, such that $\pp\in LH$ and $1<p_-\leq p_+<\infty$,
and $1<r<\infty$,
\[ \bigg\|\bigg( \sum_k (Mg_k)^r\bigg)^{\frac{1}{r}} \bigg\|_{L^\pp}
\lesssim
\bigg\|\bigg( \sum_k |g_k|^r\bigg)^{\frac{1}{r}} \bigg\|_{L^\pp}. \]
\end{lemma}

The next two results are the analogues of the previous two for the
fractional maximal operator.  The first follows from the scalar, weighted
inequalities for $M_\alpha$ and off-diagonal
extrapolation~\cite[Theorem~3.23]{MR2797562}.

\begin{lemma} \label{prop:frac-fef-stein-wts}
Given $0<\alpha<n$, $1<r<\infty$, and $1<p<\frac{n}{\alpha}$, define $q$ by
  $\frac{1}{p}-\frac{1}{q}=\frac{\alpha}{n}$.  If $w\in A_{p,q}$, then 
\[ \bigg\|\bigg( \sum_k (M_\alpha g_k)^r\bigg)^{\frac{1}{r}} \bigg\|_{L^q(w^q)} 
\lesssim
\bigg\|\bigg( \sum_k |g_k|^r\bigg)^{\frac{1}{r}} \bigg\|_{L^p(w^p)}. \]
\end{lemma}

The final inequality follows
by off-diagonal extrapolation in the variable Lebesgue
spaces~\cite[Theorem 5.28]{cruz-fiorenza-book}.   The vector-valued
inequality is not explicitly proved there; however, it can be gotten
by extrapolating starting with the weighted vector-valued inequality
in~\cite[Theorem~3.23]{MR2797562}.

\begin{lemma} \label{prop:frac-fef-stein-var}
Given $0<\alpha<n$ and $\pp \in \Pp_0$, suppose $1<p_-\leq
p_+<\frac{n}{\alpha}$.  Define $\qq$ by
  $\frac{1}{\pp}-\frac{1}{\qq}=\frac{\alpha}{n}$.  Then
\[ \bigg\|\bigg( \sum_k (M_\alpha g_k)^r\bigg)^{\frac{1}{r}} \bigg\|_{L^\qq} 
\lesssim
\bigg\|\bigg( \sum_k |g_k|^r\bigg)^{\frac{1}{r}} \bigg\|_{L^\pp}. \]
\end{lemma}

\begin{remark} \label{remark:vv-note} In applying these vector-valued
  inequalities, we will use two generalizations.  Rather than give
  these as corollaries, we instead describe the underlying ideas for
  adapting the above results.  First, since maximal operators are
  positive homogeneous, we can, for example, replace $(Mg_k)^r$ by
  $\lambda_k (Mg_k)^r=(M(\lambda_k^{\frac{1}{r}}g_k))^r$ on the
  left-hand side and $g_k^r$ by $\lambda_k g_k^r$ in the right-hand
  term.

Second, if we let $g_k=\chi_{Q_k}$ for some collection of cubes $Q_k$,
then given $0<p<\infty$,  $\tau>1$, and $w\in A_\infty$, there exists
$r>1$ such that $w\in A_{rp}$, and so we have that
\begin{equation}
\label{eq.A4010}
\bigg\|\sum_k \chi_{\tau Q_k} \bigg\|_{L^p(w)}
\lesssim
\bigg\|\bigg(\sum_k M(\chi_{ Q_k})^r\bigg)^{\frac{1}{r}}
\bigg\|_{L^{rp}(w)}^r
\lesssim 
\bigg\|\sum_k \chi_{Q_k} \bigg\|_{L^p(w)}.
\end{equation}
\end{remark}

\bigskip

We now turn to generalizations of the lemma of Grafakos and Kalton
discussed in the introduction.  First, though not actually necessary
for the proof of our main results,
we will show that by extrapolation we can easily prove a weighted
generalization of their inequality.

\begin{lemma} \label{lem:GK} For $0<p\leq 1$, if
  $w\in RH_{(\frac{1}{p})'}$, then for all sequences of cubes
  $\{Q_k\}$ and non-negative functions $\{g_k\}$ such that
  $\supp(g_k)\subset Q_k$,
\[ \bigg\| \sum_k g_k\bigg\|_{L^p(w)}
\lesssim \bigg \|\sum_k \bigg(\avgint_{Q_k} g_k\,dy\bigg) \chi_{Q_k}\bigg\|_{L^p(w)}. \]
\end{lemma}

\begin{proof} 
We first prove this for $p=1$ and $w\in RH_\infty$.  But in this case the result is trivial:
\begin{multline*}
 \bigg\| \sum_k g_k\bigg\|_{L^1(w)}
= \sum_k \int_{Q_k} g_k w\,dy  \\
\lesssim \sum_k \int_{Q_k} g_k \,dy \avgint_{Q_k} w\,dy 
= \bigg \|\sum_k \bigg(\avgint_{Q_k} g_k\,dy\bigg) \chi_{Q_k}\bigg\|_{L^1(w)}.
\end{multline*}

The proof for $0<p<1$ and $w\in RH_{(\frac{1}{p})'}$ follows at once
via reverse H\"older extrapolation, Theorem~\ref{thm:rh-ACM}. 
\end{proof}

Below, we will use the following generalization of the lemma of
Grafakos and Kalton, which lets us eliminate the hypothesis that
$p\leq 1$ but replaces the $L^1$ averages by $L^q$ averages.

\begin{lemma} \label{lemma:GK-q}
Fix $q>1$. If $0<p<q$ and $w\in RH_{(\frac{q}{p})'}$, then  for all
sequences of cubes  $\{Q_k\}$ and non-negative functions
$\{g_k\}$ such that $\supp(g_k)\subset Q_k$,  
\[ \bigg\| \sum_k g_k\bigg\|_{L^p(w)}
\lesssim \bigg \|\sum_k \bigg(\avgint_{Q_k} g_k^q\,dy\bigg)^{\frac{1}{q}} \chi_{Q_k}\bigg\|_{L^p(w)}. \]
\end{lemma}

\begin{proof} 
We first prove this for $p=1$ and $w\in RH_{q'}$.  But in this case the result is trivial:
\begin{align*}
 \bigg\| \sum_k g_k\bigg\|_{L^1(w)}
& = \sum_k \avgint_{Q_k} g_k w\,dy\; |Q_k| \\
& \leq \sum_k \bigg(\avgint_{Q_k} g_k^q \,dy\bigg)^{\frac{1}{q}}
                                              \bigg(\avgint_{Q_k} w^{q'}\,dy\bigg)^{\frac{1}{q'}} |Q_k| \\
  &  \lesssim \sum_k \bigg(\avgint_{Q_k} g_k^q
    \,dy\bigg)^{\frac{1}{q}} w(Q_k) \\
& =   \bigg \|\sum_k \bigg(\avgint_{Q_k} g_k^q\,dy\bigg)^{\frac1q} \chi_{Q_k}\bigg\|_{L^1(w)}.
\end{align*}

The desired result for $0<p<q$ and $w\in RH_{(\frac{q}{p})'}$ follows
at once from reverse H\"older extrapolation, Theorem~\ref{thm:rh-ACM}.
\end{proof}

We can now use extrapolation to extend the previous two results to the
scale of variable Lebesgue spaces.

\begin{lemma} \label{cor:var-GK} Given $\pp\in \Pp_0$, suppose
  $0<p_-\leq p_+<1$ and $\pp \in LH$.  Then for all sequences of cubes
  $\{Q_k\}$ and functions $\{g_k\}$ such that $\supp(g_k)\subset Q_k$,
\[ \bigg\| \sum_k g_k\bigg\|_{\Lp}
\lesssim \bigg \|\sum_k \bigg(\avgint_{Q_k} g_k\,dy\bigg) \chi_{Q_k}\bigg\|_{\Lp}. \]
If we only assume that $p_+<\infty$, then for any $q$ such that
$p_+<q<\infty$,
\[ \bigg\| \sum_k g_k\bigg\|_{\Lp}
\lesssim \bigg \|\sum_k \bigg(\avgint_{Q_k} g_k^q\,dy\bigg)^{\frac{1}{q}} \chi_{Q_k}\bigg\|_{\Lp}. \]
\end{lemma}

\begin{proof}
  The first inequality follows from Lemma~\ref{lem:GK} and
  Theorem~\ref{thm:rh-extrapol} with $q=1$ and any $p$ such that $0<p<p_-$.
  The second inequality follows from Lemma~\ref{lemma:GK-q} and
  Theorem~\ref{thm:rh-extrapol} with $p_+<q<\infty $ and $0<p<p_-$.
\end{proof}

\medskip

The following off-diagonal inequality plays a role in the proof of
Hardy space estimates for the fractional integral operator.  It was
first proved by Str\"omberg and Wheeden~\cite{MR766221}; here we again
give an elementary proof using extrapolation.

\begin{lemma} \label{lem:frac} Suppose $0<\alpha<n$,
  $0<p<\frac{n}{\alpha}$, and $\frac1q=\frac1p-\frac{\alpha}{n}$.  If
  $w^p\in RH_{\frac{q}{p}}$, then for any countable collection of
  cubes $\{Q_k\}$ and $\lambda_k>0$, 
\[ \left\|\sum_k
    \lambda_k|Q_k|^{\frac{\alpha}{n}}\chi_{Q_k}\right\|_{L^q(w^q)}
\lesssim 
\left\|\sum_k \lambda_k\chi_{Q_k}\right\|_{L^p(w^p)}. \]
\end{lemma}

\begin{proof}
  We will use Theorem \ref{thm:offdiagRHex} with
  $r=\frac{n}{\alpha}$, $p_0=1$, and
  $q_0=\frac{n}{n-\alpha}=(\frac{n}{\alpha})'$. We will show the
  estimate
\begin{equation}
\label{eq:L1est}
\left\|\sum_k \lambda_k
  |Q_k|^{\frac{\alpha}{n}}\chi_{Q_k}\right\|_{L^{\frac{n}{n-\alpha}}(w^{\frac{n}{n-\alpha}})}
\lesssim 
\left\|\sum_k \lambda_k \chi_{Q_k}\right\|_{L^1(w)}
\end{equation}
holds for all $w\in RH_{\frac{n}{n-\alpha}}$, all sequence of positive
numbers $\{\lambda_k\}$, and sequences of cubes $\{Q_k\}$.  If we
assume~\eqref{eq:L1est} for the moment, then by the Theorem
\ref{thm:offdiagRHex} we have that for all $0<p<\frac{n}{\alpha}$, $q$
satisfying $\frac1p-\frac1q=\frac{\alpha}{n}$, and $w$ such that
$w^p\in RH_{(\frac{n/\alpha}{p})'}=RH_{\frac{n}{n-\alpha
    p}}=RH_{\frac{q}{p}}$,
\[
  \left\|\sum_k\lambda_k|Q_k|^{\frac{\alpha}{n}}\chi_{Q_k}\right\|_{L^q(w^q)}
  \lesssim \left\|\sum_k\lambda_k\chi_k\right\|_{L^p(w^p)}, \]
which is the desired result. 

\medskip

To prove \eqref{eq:L1est}, 
let $w\in RH_{\frac{n}{n-\alpha}}$, $u=w^{\frac{n}{n-\alpha}}$, and
fix $g\geq 0$ in
$L^{(\frac{n}{n-\alpha})'}(w^{\frac{n}{n-\alpha}})=L^{\frac{n}{\alpha}}(u)$.
By duality, it will suffice to estimate the integral
$$\int_{\R^n}\left(\sum_k \lambda_k
  |Q_k|^{\frac{\alpha}{n}}\chi_{Q_k}\right)g
w^{\frac{n}{n-\alpha}}\,dx 
=
\sum_k \lambda_k |Q_k|^{\frac{\alpha}{n}}\int_{Q_k}g u\,dx.$$
But then we have that
\begin{align*}
\sum_k \lambda_k |Q_k|^{\frac{\alpha}{n}}\int_{Q_k}g u\,dx
&\leq 
\sum_k \lambda_k |Q_k|^{\frac{\alpha}{n}}\left(\int_{Q_k}g^{\frac{n}{\alpha}} u\,dx\right)^{\frac{\alpha}{n}} u(Q_k)^{1-\frac{\alpha}{n}}\\ 
&\leq 
\|g\|_{L^{\frac{n}{\alpha}}(u)}\sum_k \lambda_k|Q_k|^{\frac{\alpha}{n}} \left(\int_{Q_k}w^{\frac{n}{n-\alpha}}\,dx\right)^{1-\frac{\alpha}{n}} \\
&\leq 
C\|g\|_{L^{\frac{n}{\alpha}}(u)}\sum_k \lambda_k|Q_k|^{\frac{\alpha}{n}} \int_{Q_k}w\,dx |Q_k|^{-\frac{\alpha}{n}}\\
&=
C\|g\|_{L^{\frac{n}{\alpha}}(u)}\sum_k \lambda_k \int_{Q_k}w\,dx\\
&=
C\|g\|_{L^{\frac{n}{\alpha}}(u)}\int_{\R^n}\left(\sum_k \lambda_k
  \chi_{Q_k}\right)w\,dx.\\
\end{align*}
This completes the proof.
\end{proof}

Our final estimate extends Lemma~\ref{lem:frac} to the variable Lebesgue
spaces.  It was first proved by Sawano~\cite{MR3090168}; however, it
follows immediately from Lemma~\ref{lem:frac} by extrapolation, Theorem~\ref{thm:offdiagrh-extrapol}.

\begin{lemma} \label{lemma:frac-var}
Given $0<\alpha<n$, suppose $\pp\in \Pp_0$ is such that $\pp \in LH$
and $0<p_-\leq p_+< \frac{n}{\alpha}$.  Define $\qq$ by
$\frac{1}{\pp}-\frac{1}{\qq}=\frac{\alpha}{n}$.  Then for any countable
collection of cubes $\{Q_k\}$ and $\lambda_k>0$,
\[ \bigg\|\sum_k \lambda_k |Q_k|^{\frac{\alpha}{n}}\chi_{Q_k}
  \bigg\|_{\qq}
  \lesssim
  \bigg\|\sum_k \lambda_k \chi_{Q_k}
  \bigg\|_{\pp}. \]
\end{lemma}

\section{Singular integral operators}
\label{section:sio}

In this section we prove Theorems~\ref{thm:sio-weight}
and~\ref{thm:sio-var}.  For the proof we need two lemmas; the
essential ideas in their proofs are well-known (see, for
instance,~\cite{stein93}) but to get the versions we need--which will
be applicable to both singular integrals and fractional integrals--we
give their short proofs.  Throughout this section and
Sections~\ref{section:frac-int} and~\ref{section:nonconv} below, let
$\phi \in C_c^\infty$ be a fixed function supported in $B(0,1)$ with
$\int \phi\,dx=1$.

\begin{lemma} \label{lemma:kernel-est}
Fix $N \geq 0$ and $0\leq \alpha <n$.  Let $K$ be a distribution such that
$|\hat{K}(\xi)|\lesssim |\xi|^{-\alpha}$.  Suppose further that away
from the origin $K$ agrees with a function in $C^{N+1}$, and for all multi-indices
$\beta$ such that $|\beta|\leq N+1$,
\[ |\partial^\beta K(x)| \leq \frac{B_0}{|x|^{n-\alpha+|\beta|}}. \]
  Define $K^t = \phi_t* K$.  Then $K^t$ is a smooth function that
  satisfies
\[ |\partial^\beta K^t(x)| \leq \frac{B_1}{|x|^{n-\alpha+|\beta|}}. \]
  The constant $B_1$ is independent of $t$.
\end{lemma}

\begin{proof}
  Fix $t>0$.  Then we have that
  \[ K^t(x)  = \int_{B(0,t)} \phi_t(y) K(x-y)\,dy.  \]
  Suppose first that $|x|>2t$.  Then on $B(0,t)$ both functions
  and their derivatives are continuous and bounded, and so we can take the derivative inside the integral
  to get
  \begin{multline*}
    |\partial^\beta K^t(x)|
    \leq
    \int_{B(0,t)} t^{-n} |\phi(y/t)|
    |\partial^\beta K(x-y)|\,dy \\
    \lesssim
    t^{-n} \int_{B(0,t)} |x-y|^{-n+\alpha-|\beta|}\,dy
    \lesssim |x|^{-n+\alpha-|\beta|}. 
  \end{multline*}
  
  If $|x|\leq 2t$, then by the inverse Fourier transform
  \[ \partial^\beta K^t(x)
    \approx
    \int_{\R^n} e^{-2\pi i x\cdot \xi}
    \xi^\beta \hat{K}(\xi) \hat{\phi}(t\xi)\,d\xi, \]
  and so
  \begin{multline*}
    |\partial^\beta K^t(x)|
    \lesssim
    t^{\alpha-|\beta|} \int_{\R^n} |t\xi|^{-\alpha+|\beta|}
    |\hat{\phi}(t\xi)|\,d\xi\\
    \lesssim 
    t^{-n +\alpha-|\beta|} \int_{\R^n}
    |u|^{-\alpha+|\beta|}|\hat{\phi}(u)|\,du
    \lesssim
    |x|^{-n +\alpha-|\beta|}. 
  \end{multline*}
  The final integral converges since $\alpha<n$ and since $\hat{\phi}$
  is a Schwartz function.  
\end{proof}

\begin{lemma} \label{lemma:sio-tail}
Let $N$, $\alpha$ and $K$ be as in Lemma~\ref{lemma:kernel-est} and
define the operator $T$ by $Tf=K*f$.   Let $a$ be any $(N,\infty)$ atom with
$\supp(a)\subset Q$.  Then for all $x\in (Q^*)^c$,
\begin{equation}
\label{eq.5001}
M_\phi (Ta)(x) \lesssim
  M_{\alpha_\tau}(\chi_Q)(x)^{\tau},
\end{equation}
where $\tau=\frac{n+N+1}{n}$ and $\alpha_\tau=\alpha/\tau$.  
\end{lemma}

\begin{proof}
  Fix $x\in (Q^*)^c$ and $t>0$.  Then it will suffice
  to show that
  \[ |\phi_t*Ta(x)|
    \lesssim
   M_{\alpha_\tau}(\chi_Q)(x)^{\tau}, \]
  with a constant independent of $x$ and $t$.  Define $K^t=\phi_t*K$
  as before.
  Since
  \[ Ta(x) = \int_Q K(x-y)a(y)\,dy \]
  and this integral converges absolutely, by taking the Fourier
  transform we see that $\phi_t*Ta(x)=K^t*a(x)$.    Let $P_N$ be the
  Taylor polynomial of degree $N$ of the function $y\rightarrow K^t(x-y)$ centered
  at $c_Q$.  Then
  \[ P_N(y) = \sum_{|\beta|\leq N} \frac{\partial^\beta
      K^t(x-c_Q)}{\beta!}(y-c_Q)^\beta, \]
so  by the moment condition on $a$,
  \[ \int_Q P_N(y)a(y)\,dy = 0. \]
Moreover, the remainder $R_N(y)$ satisfies
\[ R_N(y)=  K^t(x-y)-P_N(y)
  = \sum_{|\beta|=N+1} R_\beta(y) (y-c_Q)^\beta, \]
where
\[ R_\beta (y) = \int_0^1 (1-s)^N \partial^\beta
  K^t(x-c_Q-s(y-c_Q))\,ds.  \]
Therefore, by Lemma~\ref{lemma:kernel-est}
we have that
  \begin{multline*}
    |K^t*a(x)|
    \leq
    \int_Q |K^t(x-y)-P_N(y)||a(y)|\,dy \\
    \lesssim
    \frac{1}{|x-c_Q|^{n-\alpha+N+1}}\int_Q |y-c_Q|^{N+1} \,dy 
    \lesssim
 \frac{\ell(Q)^{n+N+1}}{|x-c_Q|^{n-\alpha+N+1}}.
\end{multline*}

  To complete the proof, note that for $x\in (Q^*)^c$, if $P$ is the
  smallest cube containing $x$ and $Q$, then 
  \[ M_{\alpha_\tau}(\chi_Q)(x)^\tau
    \approx
    \bigg[|P|^{\frac{\alpha_\tau}{n}}\avgint_P \chi_Q\,dy\bigg]^\tau
    \approx
   \bigg[ \frac{\ell(Q)^n}{|x-c_Q|^{n-\alpha_\tau}}\bigg]^\tau =
   \frac{\ell(Q)^{n+N+1}}{|x-c_Q|^{n-\alpha+N+1}}.  \]
If we combine these estimates we get the desired inequality.
\end{proof}

\begin{proof}[Proof of Theorem~\ref{thm:sio-weight}]
  By the finite atomic decomposition,
  Proposition~\ref{prop:finite-atomic-wts}, it will suffice to fix a
  finite sum of $(N,\infty)$ atoms,
\[ f = \sum_{i=1}^M \lambda_i a_i, \]
with $\supp(a_i) \subset Q_i$ and $c_i=c_{Q_i}$, and prove that
\[ \|M_\phi Tf\|_{L^p(w)}
  \lesssim
  \bigg\|\sum_{i=1}^M \lambda_i \chi_{Q_i}\bigg\|_{L^p(w)}. \]
By the linearity of $T$ and the sublinearity of $M_\phi$,
\[ \|M_\phi Tf\|_{L^p(w)}
  \leq
  \bigg\|\sum_{i=1}^M \lambda_i M_\phi T(a_i)\chi_{Q_i^*}\bigg\|_{L^p(w)}
  +
  \bigg\|\sum_{i=1}^M \lambda_i M_\phi T(a_i)\chi_{(Q_i^*)^c}\bigg\|_{L^p(w)}
  = I_1 + I_2. \]

To estimate $I_1$, we apply Lemma~\ref{lemma:GK-q}.  Since $w\in
A_\infty$, $w\in RH_s$ for some $s>1$.  Fix $q>\max(p,1)$ such that
$(\frac{q}{p})'\leq s$.  Then $w\in RH_{(\frac{q}{p})'}$, and so by
Lemma~\ref{lemma:GK-q} and the fact that $T$ and $M_\phi$ are bounded
on $L^q$,
\begin{multline*}
  I_1
 \lesssim
    \bigg\|\sum_{i=1}^M \lambda_i
    \bigg(\avgint_{Q_i} M_\phi
    T(a_i)^q\,dx\bigg)^{\frac{1}{q}}\chi_{Q_i^*}\bigg\|_{L^p(w)} \\
 \lesssim
    \bigg\|\sum_{i=1}^M \lambda_i
    \bigg(\avgint_{Q_i} 
    |a_i|^q\,dx\bigg)^{\frac{1}{q}}\chi_{Q_i^*}\bigg\|_{L^p(w)} 
 \lesssim 
    \bigg\|\sum_{i=1}^M \lambda_i \chi_{Q_i^*}\bigg\|_{L^p(w)} 
 \lesssim 
    \bigg\|\sum_{i=1}^M \lambda_i \chi_{Q_i}\bigg\|_{L^p(w)};
  \end{multline*}
the last inequality follows by Lemma~\ref{lemma:fef-stein-wts} and
Remark~\ref{remark:vv-note}. 

\medskip

To estimate $I_2$, first note that by our assumption on $N$,
\[ N+1 > n\left(\frac{r_w}{p}-1\right), \]
or equivalently,
\[ p\tau = p\left(\frac{n+N+1}{n}\right) > r_w.  \]
Therefore, $w\in A_{p\tau}$, and so by Lemma~\ref{lemma:sio-tail} and
Lemma~\ref{lemma:fef-stein-wts},
\begin{multline*}
  I_2
 \lesssim
    \bigg\| \sum_{i=1}^M \lambda_i M(\chi_{Q_i})^{\tau}\bigg\|_{L^p(w)} 
 =
    \bigg\| \bigg(\sum_{i=1}^M \lambda_i
    M(\chi_{Q_i})^{\tau}\bigg)^\frac{1}{\tau}\bigg\|_{L^{p\tau}(w)}^\tau \\
 \lesssim
    \bigg\| \bigg(\sum_{i=1}^M \lambda_i
    \chi_{Q_i}\bigg)^\frac{1}{\tau}\bigg\|_{L^{p\tau}(w)}^\tau 
 = \bigg\|\sum_{i=1}^M \lambda_i \chi_{Q_i}\bigg\|_{L^{p\tau}(w)}. 
\end{multline*}
This completes the proof.
\end{proof}

\begin{proof}[Proof of Theorem~\ref{thm:sio-var}]
  The proof of this result is nearly identical to the above proof.  By
  Proposition~\ref{prop:finite-atomic-var} we may again consider
 finite sums of atoms.  We decompose as before into $I_1$ and $I_2$.  To
  estimate $I_1$ we fix $q>\max(p_+,1)$ and apply
  Lemma~\ref{cor:var-GK}.  We then use
  Lemma~\ref{prop:fef-stein-var} and argue as in Remark~\ref{remark:vv-note}.   To estimate $I_2$, we note that
  $p_-\tau>1$, and so we can again apply Lemma~\ref{prop:fef-stein-var}.
\end{proof}

\section{Fractional integral operators}
\label{section:frac-int}

In this section we prove Theorems~\ref{thm:fracint:wtd}
and~\ref{thm:fracint-var}.   The proof of
Theorem~\ref{thm:fracint:wtd} is very similar to the proof of
Theorem~\ref{thm:sio-weight} and so we will omit those details that
are the same and concentrate on the differences.  And again, the proof
of Theorem~\ref{thm:fracint-var} is a straightforward variation of the
the proof of Theorem~\ref{thm:fracint:wtd}.

\begin{proof}[Proof of Theorem~\ref{thm:fracint:wtd}]
  We need to show that if $f$ is a finite sum of $(N,\infty)$ atoms,
  \[ f = \sum_{i=1}^M \lambda_i a_i, \]
  where the exact value of $N$ will be chosen below, then
  \[ \|M_\phi I_\alpha f \|_{L^q(w^q)}
    \lesssim 
    \bigg\| \sum_{i=1}^M \lambda_i \chi_{Q_i}\bigg\|_{L^p(w^p)}.  \]
  As before we dominate the left-hand side by the sume of two terms:
  \[ \bigg\| \sum_{i=1}^M \lambda_i M_\phi I_\alpha
    (a_i)\chi_{Q_i^*}\bigg\|_{L^q(w^q)}
    +
     \bigg\| \sum_{i=1}^M \lambda_i M_\phi I_\alpha
     (a_i)\chi_{(Q_i^*)^c}\bigg\|_{L^q(w^q)}
     =
     J_1 + J_2.  \]

   To estimate $J_1$, fix $q_0> \max\big(q, \frac{n}{n-\alpha}\big)$
   and define $p_0>1$ by
   $\frac{1}{p_0}-\frac{1}{q_0}=\frac{\alpha}{n}$.  Since $w^p \in
   RH_{\frac{q}{p}}$, $w^q\in A_\infty$, so arguing as before we may assume
   that $q_0$ is such that $w^q \in RH_{(\frac{q_0}{p})'}$.  Then by
   Lemma~\ref{lemma:GK-q}, and since $M_\phi$ is bounded on $L^{q_0}$ and
   $I_\alpha : L^{p_0}\rightarrow L^{q_0}$,
   \begin{align*}
     J_1
     & \lesssim
       \bigg\| \sum_{i=1}^M \lambda_i \bigg( \avgint_{Q_i} M_\phi I_\alpha
    (a_i)^{q_0}\,dx\bigg)^{\frac{1}{q_0}}\chi_{Q_i^*}\bigg\|_{L^q(w^q)}
     \\
     & \lesssim
       \bigg\| \sum_{i=1}^M \lambda_i |Q_i|^{\frac{\alpha}{n}}\bigg( \avgint_{Q_i} 
    |a_i|^{p_0}\,dx\bigg)^{\frac{1}{p_0}}\chi_{Q_i^*}\bigg\|_{L^q(w^q)}
     \\
     & \lesssim
       \bigg\| \sum_{i=1}^M \lambda_i |Q_i^*|^{\frac{\alpha}{n}}
       \chi_{Q_i^*}\bigg\|_{L^q(w^q)}; \\
     \intertext{by Lemma~\ref{lem:frac}, }
   & \lesssim
       \bigg\| \sum_{i=1}^M \lambda_i 
     \chi_{Q_i^*}\bigg\|_{L^p(w^p)}; \\
     \intertext{since $w^p\in A_\infty$, by Lemma~\ref{lemma:fef-stein-wts}}
    & \lesssim
       \bigg\| \sum_{i=1}^M \lambda_i 
     \chi_{Q_i}\bigg\|_{L^p(w^p)}. \\ 
   \end{align*}

   \medskip
   
   To estimate $J_2$ we will apply Lemma~\ref{lemma:sio-tail}, but first we
   need to fix $N$.  For $I_\alpha$, our kernel is
   $K(x)=|x|^{\alpha-n}$, and so the desired estimates on the
   derivative of $K$ hold for all $N>0$.   We now fix $N$ as follows:
   since $w^q \in A_\infty$, choose $N$ so that
   \[ \bigg(\frac{n-\alpha+N+1}{n}\bigg)q >r_{w^q}.  \]
   As before, let $\tau = \frac{n+N+1}{n}$.  Then, since $\frac{1}{\tau p}-\frac{1}{\tau q} = \frac{\alpha}{\tau n}$, we have that
   \[ 1+ \frac{\tau q}{(\tau p)'} = \tau q \bigg(1- \frac{\alpha}{\tau
       n}\bigg)
     =  \bigg(\frac{n-\alpha+N+1}{n}\bigg)q.  \]
Hence, if we let $v=w^{\frac{1}{\tau}}$, we have that
   $v^{\tau  q} = w^q \in A_{1+ \frac{\tau q}{(\tau p)'}}$.  Equivalently, we
     have that $v\in A_{\tau p, \tau q}$.   Therefore, by
     Lemma~\ref{lemma:sio-tail} and by Lemma~\ref{prop:frac-fef-stein-wts}
     applied to the fractional maximal
     operator $M_{\alpha_\tau}$,
     \begin{multline*}
       J_2
       \lesssim
       \bigg\|\bigg(\sum_{i=1}^M \lambda_i
       M_{\alpha_\tau}(\chi_{Q_i})^\tau\bigg)^{\frac{1}{\tau}}
       \bigg\|_{L^{q\tau}(v^{\tau q})}^\tau \\
       \lesssim
       \bigg\|\bigg(\sum_{i=1}^M \lambda_i
    \chi_{Q_i}\bigg)^{\frac{1}{\tau}}
    \bigg\|_{L^{p\tau}(v^{\tau p})}^\tau
    =
    \bigg\|\sum_{i=1}^M \lambda_i
    \chi_{Q_i}
    \bigg\|_{L^{p}(w^{ p})}.
  \end{multline*}
  This completes the proof.
\end{proof}

\begin{proof}[Proof of Theorem~\ref{thm:fracint-var}]
The proof is essentially the same as the proof of
Theorem~\ref{thm:fracint:wtd}.  To estimate $J_1$ we use
Lemma~\ref{cor:var-GK} with $q_0>\max\big(q_+,
\frac{n}{n-\alpha}\big)$, and Lemma~\ref{lemma:frac-var}.  To
estimate $J_2$ we choose $N$ so large that $p_-\tau >1$ so that we can
apply Lemma~\ref{prop:frac-fef-stein-var} to $M_{\alpha_\tau}$ acting from
$L^{\tau\pp}$ to $L^{\tau\qq}$.
\end{proof}

\section{Non-convolution operators}
\label{section:nonconv}

The proofs of Theorems~\ref{thm:wtd-nonconv} and~\ref{thm:var-nonconv}
for non-convolution Calder\'on-Zygmund operators are essentially
identical to the proofs for convolution type singular integrals in
Theorems~\ref{thm:sio-weight} and~\ref{thm:sio-var}.  Fix such an operator
$T$.  Since it satisfies the standard kernel estimates~\eqref{eq.8001}
and~\eqref{eq.8002}, it is bounded on $L^q$, $1<q<\infty$.  Therefore,
the estimate of the local piece is identical.  To prove the estimate
for the global piece, we need a maximal operator estimate for the
action of $T$ on atoms, which is the substance of
Lemma~\ref{lemma:nonconv-max} below.  The rest of the proof is
identical.

\begin{remark}
  Our results can be generalized to a larger class
of operators.  An examination of the proof shows that to apply
Lemma~\ref{lemma:GK-q} and estimate the local piece, we need to assume
that the operator $T$ is such that there exists $q>\max(p,1)$ such
that $w\in RH_{(\frac{q}{p})'}$ (in the weighted case) and $T$ satisfies
$\|Ta\|_q \lesssim |Q|^{\frac{1}{q}}$ (in either the weighted or
variable exponent case).  To estimate the global
piece we again need that $T$ has a kernel $K$ representation such that 
\eqref{eq.8003} holds for all $|\beta|\le L+1$ and $T$ satisfies
\eqref{eq.8004}  for all $|\beta|\le L$, where $L$ is defined as
before in Theorem~\ref{thm:wtd-nonconv} (for a weighted estimate) or
Theorem~\ref{thm:var-nonconv} (for the variable exponent case).  In
particular, we do not assume that the operator is bounded on $L^2$ or
that it satisfies the standard kernel estimates~\eqref{eq.8001}
and~\eqref{eq.8002}.   We leave the details to the interested reader.
\end{remark}

\begin{lemma} \label{lemma:nonconv-max}
Given $L \geq -1$, suppose  $T$ is a Calder\'on-Zygmund operator associated with a kernel
$K$ that satisfies \eqref{eq.8003} for all $|\beta|=L+1$, and suppose
$T$ satisfies \eqref{eq.8004} for all $(L+1,\infty)$ atoms and
$|\beta|\leq L$.  (If $L=-1$ we disregard this condition.)  Then given
any $(L+1,\infty)$ atom $a$, $\supp(a)\subset Q$, then for all $x\in
(Q^{**})^c$, and $\phi$ as defined in Section~\ref{section:sio},
\begin{equation} \label{eqn:nonconv-max1}
  M_\phi(Ta)(x) \lesssim M(\chi_Q)(x)^{\frac{n+L+1}{n}}. 
\end{equation}
\end{lemma}

\begin{proof}
Fix $x\in (Q^{**})^c$; then to prove~\eqref{eqn:nonconv-max1} it
will suffice to prove that for all $t>0$, 
\begin{equation}
\label{eq.8008}
\big|
\phi_t*Ta(x)
\big|
\lesssim
\frac{\ell(Q)^{n+L+1}}{|x-c|^{n+L+1}},
\end{equation}
where the implicit constant is independent of $t$, $x$ and $Q$.
We will consider two cases: $0<t\le \frac{|x-c|}{2}$ and $t>\frac{|x-c|}{2}$.

First, however, we will estimate the decay of $Ta(y)$ when  $y\in
(Q^*)^c$.  Let $N=L+1$ and let $c$ be the center of $Q$.  By our assumption on the atom $a$ we have
that $\int z^\beta a(z)dz=0$ for all $|\beta|\le N$.  Hence, we can
apply Taylor's theorem with integral remainder to get
\begin{align}
|Ta(y)| = &
\bigg|
\int_Q K(y,z)a(z)\,dz
\bigg|\notag \\
=&
\bigg|
\int_Q \bigg[
K(y,z)
-\sum_{|\beta|\le N-1}\frac{\partial_z^\beta K(y,c)}{\beta!}(z-c)^\beta
\bigg]a(z)\,dz
\bigg|\notag \\
=&
\bigg|
\sum_{|\beta|= N}
\frac{|\beta|}{\beta!}
\int_Q \bigg(
\int_0^1(1-\theta)^{N-1}
\partial_z^\beta K(y,\zeta_{z,\theta})
d\theta
\bigg)
(z-c)^\beta a(z)\,dz
\bigg|,\notag \\
\intertext{where $\zeta_{z,\theta} = c+\theta(z-c)$.  We again apply
the vanishing moment condition of $a$ and  \eqref{eq.8003} with $|\beta|=N$ to get}
=&
\bigg|
\sum_{|\beta|= N}
\frac{|\beta|}{\beta!}
\int_0^1(1-\theta)^{N-1}
\int
\bigg(
\partial_z^\beta K(y,\zeta_{z,\theta})
-
\partial_z^\beta K(y,c)
\bigg)
(z-c)^\beta a(z)\,dz
d\theta
\bigg|\notag \\
\lesssim&
\int
\frac{|z-c|^{\delta}}{|y-c|^{n+N+\delta}}
|z-c|^N |a(z)|\,dz \notag \\
\lesssim &
\frac{\ell(Q)^{n+N+\delta}}{|y-c|^{n+N+\delta}}. \label{eq.8007}
\end{align}

\medskip

We now consider the two cases given above.  First, suppose $0<t\le
\frac{|x-c|}{2}$.  Then we  have that if $|x-y|\leq t$, 
\[ |y-c|\ge |x-c|-|x-y|\ge |x-c|-t\ge \frac12|x-c|,  \]
which, since $x\in (Q^{**})^c$, implies that  $y\in
(Q^*)^c$. Therefore, by \eqref{eq.8007} we have that
\begin{equation*}
\label{eq.8009}
|\phi_t*Ta(x)| 
\le
\int_{|x-y|\le t}
|\phi_t(x-y)|
|Ta(y)|dy
\lesssim
\frac{\ell(Q)^{n+N+\delta}}{|x-c|^{n+N+\delta}} 
\leq
\frac{\ell(Q)^{n+L+1}}{|x-c|^{n+L+1}}.
\end{equation*}
The last inequality holds since $\delta>0$ and $\ell(Q)\leq |x-c|$.
This gives us \eqref{eq.8008} in the first case.

\medskip

Now suppose that $t>\frac{|x-c|}{2}$.  Then by the moment condition
for $Ta$ \eqref{eq.8004}, and again by Taylor's theorem, we have that
\begin{align*}
|\phi_t*Ta(x)| =& \bigg|\int_Q
                  t^{-n}\phi\big(t^{-1}(x-y)\big)Ta(y)\,dy \bigg|\\
=&\bigg|\int_{\R^n}  t^{-n}
\bigg[
\phi\big(t^{-1}(x-y)\big)
-\sum_{|\beta|\le L}\frac{\partial^\beta 
\phi(\frac{x-c}{t})}{\beta!}\bigg(\frac{c-y}{t}\bigg)^\beta
\bigg]
Ta(y)\,dy \bigg|\\
=&
\bigg|\int_{\R^n} 
\sum_{|\beta|= N}
\frac{|\beta|}{t^{n}\beta!}
\bigg(
\int_0^1
(1-\theta)^{L}
\partial^\beta \phi\big(\frac{\zeta_{x,y,\theta}}t\big)
\,d\theta\bigg)
\bigg(\frac{c-y}{t}\bigg)^\beta
   Ta(y)\,dy \bigg|, \\
\intertext{where $\zeta_{x,y,\theta} = x-c+\theta(c-y)$.  Continuing,
  we get}
\lesssim&
\sum_{|\beta|= N}t^{-n-N}
\int_{\mathbb{R}^n}
\sup_{0\le \theta\le 1}\big|
\partial^\beta \phi\big(\frac{\zeta_{x,y,\theta}}t\big)
\big|
|c-y|^{N}
|Ta(y)|\,dy\\
\lesssim&
|x-c|^{-n-N}
\int_{\mathbb{R}^n}
|y-c|^{N}|Ta(y)|\,dy.
\end{align*}

To estimate the final integral we split the domain.  
On $(Q^*)^c$ we use \eqref{eq.8007} to get
\[
\int_{(Q^*)^c}
|y-c|^{N}|Ta(y)|dy
\lesssim
\int_{(Q^*)^c}
\frac{\ell(Q)^{n+N+\delta}}{|y-c|^{n+\delta}}dy
\lesssim
\ell(Q)^{n+N}.
\]
On the other hand, to estimate the integral on $Q^*$ we use the fact
that  $T$ is bounded on $L^r$ for $r>1$.  Then
by H\"older's inequality we have that
\[
\int_{Q^*}
|y-c|^{N}|Ta(y)|\,dy
\lesssim
\ell(Q)^{n/r'+N}\|Ta\|_{L^r}
\lesssim
\ell(Q)^{n+N}.
\]
If we combine all of these estimates, we see that 
for $t>\frac{|x-c|}{2}$,
\begin{equation*}
\big|\phi_t*Ta(x)\big|
\lesssim
\frac{\ell(Q)^{n+L+1}}{|x-c|^{n+L+1}},
\end{equation*}
which give us \eqref{eq.8008} in this case.  This completes the proof.
\end{proof}

\section{Extensions to other Banach function spaces}
\label{section:BFS}

In this section we conclude by briefly considering the extension of our
approach to Hardy spaces defined with
respect to other quasi-Banach function spaces.  Our starting point is
the observation that we were able to prove 
results for the
variable Hardy spaces because we could prove vector-valued inequalities in the variable
Lebesgue spaces in Section~\ref{section:lemmas} via extrapolation from
the corresponding weighted norm inequalities.

Therefore, to extend our results to other scales of spaces we need a
theory of extrapolation.  
In~\cite{MR2797562},   motivated by the  extrapolation results to the scale of
variable Lebesgue spaces
in~\cite{cruz-uribe-fiorenza-martell-perez06}, the
authors considered the  general problem of extrapolating from weighted
norm inequalities into
quasi-Banach function spaces.    Their approach was the following:  given a
a quasi-Banach function space  $X$, define a scale of spaces $X^r$, $0<r<\infty$,
where  $f\in X^r$ if $|f|^r \in X$, and the ``norm'' on $X^r$ is given
by
\[ \|f\|_{X^r} = \||f|^r\|_X^{\frac{1}{r}}. \]
In order to use extrapolation to prove results in $X$, it is necessary
to assume 
that there exists $r>1$ such that $X^r$ is a Banach function space,
and that the Hardy-Littlewood maximal operator is bounded on the
associate space $(X^r)'$.  (See~\cite[Remark~4.7]{MR2797562}.)   Thus,
for example, in the scale of variable Lebesgue spaces, given $\pp \in \Pp_0$, we
would fix $0<p<p_-$ and use the fact that if $\pp\in LH$, the maximal
operator is bounded on the Banach function space $L^{(\pp/p)'}$.  

Given a quasi-Banach space $X$, we can define a Hardy space $H^X$ to
be the set of all distributions $f$ such that $\M_{N_0}f \in X$, with
quasi-norm $\|f\|_{H^X} = \|\M_{N_0}f \|_X$.  The question is then
whether we can prove that singular and fractional integrals are
bounded on the spaces $H^X$.  The proof would require two components.
First, we would need the theory of extrapolation to prove the various
vector-valued inequalities required.  Second, we would need the basics
Hardy space theory: in particular, the equivalence of the various
definitions of a Hardy space in terms of the radial and grand maximal
operators, and the finite atomic decomposition.

We could, for instance, apply these ideas to the Hardy-Orlicz spaces introduced by
Janson~\cite{MR596123}, and considered earlier in the case of analytic
functions on the unit disk by Le\'{s}niewicz~\cite{MR0215072}.  An
atomic decomposition for these spaces was given by
Viviani~\cite{MR996824}.  Another, more recent example of spaces that
would be amenable to our approach  are the
Musielak-Orlicz Hardy spaces.  These are based on the Musielak-Orlicz
spaces introduced in~\cite{musielak83}; see~\cite{MR3586020} for a 
comprehensive treatment, including an atomic decomposition.  (We note
in passing that extrapolation into the scale of Musielak-Orlicz spaces
was considered separately in~\cite{MR3811530}.)  Additional examples
Hardy spaces where our approach might be applicable are given in~\cite{MR3687096}.

\bibliographystyle{plain}

\bibliography{GK}

\begin{thebibliography}{10}

\bibitem{andersen-john80}
K.~F. Andersen and R.~T. John.
\newblock Weighted inequalities for vector-valued maximal functions and
  singular integrals.
\newblock {\em Studia Math.}, 69(1):19--31, 1980/81.

\bibitem{MR3785798}
T.~Anderson, D.~Cruz-Uribe, and K.~Moen.
\newblock Extrapolation in the scale of generalized reverse {H}\"{o}lder
  weights.
\newblock {\em Rev. Mat. Complut.}, 31(2):263--286, 2018.

\bibitem{MR518170}
R.~Coifman and Y.~Meyer.
\newblock {\em Au del\`a des op\'erateurs pseudo-diff\'erentiels}, volume~57 of
  {\em Ast\'erisque}.
\newblock Soci\'et\'e Math\'ematique de France, Paris, 1978.
\newblock With an English summary.

\bibitem{cruz-fiorenza-book}
D.~Cruz-Uribe and A.~Fiorenza.
\newblock {\em Variable Lebesgue Spaces: Foundations and Harmonic Analysis}.
\newblock Birkh\"auser, Basel, forthcoming.

\bibitem{cruz-uribe-fiorenza-martell-perez06}
D.~Cruz-Uribe, A.~Fiorenza, J.~M. Martell, and C.~P{\'e}rez.
\newblock The boundedness of classical operators on variable {$L\sp p$} spaces.
\newblock {\em Ann. Acad. Sci. Fenn. Math.}, 31(1):239--264, 2006.

\bibitem{MR3811530}
D.~Cruz-Uribe and P.~H\"{a}st\"{o}.
\newblock Extrapolation and interpolation in generalized {O}rlicz spaces.
\newblock {\em Trans. Amer. Math. Soc.}, 370(6):4323--4349, 2018.

\bibitem{MR3788859}
D.~Cruz-Uribe and J.~M. Martell.
\newblock Limited range multilinear extrapolation with applications to the
  bilinear {H}ilbert transform.
\newblock {\em Math. Ann.}, 371(1-2):615--653, 2018.

\bibitem{MR2797562}
D.~Cruz-Uribe, J.~M. Martell, and C.~P{\'e}rez.
\newblock {\em Weights, extrapolation and the theory of {R}ubio de {F}rancia},
  volume 215 of {\em Operator Theory: Advances and Applications}.
\newblock Birkh\"auser/Springer Basel AG, Basel, 2011.

\bibitem{1708.07195}
D.~Cruz-Uribe, K.~Moen, and H.~V. Nguyen.
\newblock The boundedness of multilinear {C}alderón-{Z}ygmund operators on
  weighted and variable {H}ardy spaces.
\newblock {\em Publ. Mat.}, to appear.

\bibitem{DCU-dw-P2014}
D.~Cruz-Uribe and L.-A. Wang.
\newblock Variable {H}ardy spaces.
\newblock {\em Indiana Univ. Math. J.}, 63(2):447--493, 2014.

\bibitem{diening-harjulehto-hasto-ruzicka2010}
L.~Diening, P.~Harjulehto, P.~H\"ast\"o, and M.~R{\r{u}}{\v{z}}i{\v{c}}ka.
\newblock {\em Lebesgue and {S}obolev spaces with {V}ariable {E}xponents},
  volume 2017 of {\em Lecture Notes in Mathematics}.
\newblock Springer, Heidelberg, 2011.

\bibitem{duoandikoetxea01}
J.~Duoandikoetxea.
\newblock {\em Fourier analysis}, volume~29 of {\em Graduate Studies in
  Mathematics}.
\newblock American Mathematical Society, Providence, RI, 2001.

\bibitem{fefferman-stein71}
C.~Fefferman and E.~M. Stein.
\newblock Some maximal inequalities.
\newblock {\em Amer. J. Math.}, 93:107--115, 1971.

\bibitem{MR549091}
J.~Garc\'{i}a-Cuerva.
\newblock Weighted {$H^{p}$} spaces.
\newblock {\em Dissertationes Math. (Rozprawy Mat.)}, 162:63, 1979.

\bibitem{garcia-cuerva-rubiodefrancia85}
J.~Garc{\'{\i}}a-Cuerva and J.~L. Rubio~de Francia.
\newblock {\em Weighted norm inequalities and related topics}, volume 116 of
  {\em North-Holland Mathematics Studies}.
\newblock North-Holland Publishing Co., Amsterdam, 1985.

\bibitem{MR784004}
A.~E. Gatto, C.~E. Guti\'{e}rrez, and R.~L. Wheeden.
\newblock Fractional integrals on weighted {$H^p$} spaces.
\newblock {\em Trans. Amer. Math. Soc.}, 289(2):575--589, 1985.

\bibitem{MR3525560}
L.~Grafakos and D.~He.
\newblock Weak {H}ardy spaces.
\newblock In {\em Some topics in harmonic analysis and applications}, volume~34
  of {\em Adv. Lect. Math. (ALM)}, pages 177--202. Int. Press, Somerville, MA,
  2016.

\bibitem{MR1852036}
L.~Grafakos and N.~Kalton.
\newblock Multilinear {C}alder\'{o}n-{Z}ygmund operators on {H}ardy spaces.
\newblock {\em Collect. Math.}, 52(2):169--179, 2001.

\bibitem{GNNS-preprint}
L.~Grafakos, S.~Nakamura, H.~V. Nguyen, and Y.~Sawano.
\newblock Multiplier conditions for boundedness into {H}ardy spaces.
\newblock {\em Ann. Inst. Fourier (Grenoble)}, to appear.
\newblock arXiv:1702.08190v1.

\bibitem{MR3448918}
J.~Hart and G.~Lu.
\newblock Hardy space estimates for {L}ittlewood-{P}aley-{S}tein square
  functions and {C}alder\'{o}n-{Z}ygmund operators.
\newblock {\em J. Fourier Anal. Appl.}, 22(1):159--186, 2016.

\bibitem{MR3649238}
J.~Hart and L.~Oliveira.
\newblock Hardy space estimates for limited ranges of {M}uckenhoupt weights.
\newblock {\em Adv. Math.}, 313:803--838, 2017.

\bibitem{MR596123}
S.~Janson.
\newblock Generalizations of {L}ipschitz spaces and an application to {H}ardy
  spaces and bounded mean oscillation.
\newblock {\em Duke Math. J.}, 47(4):959--982, 1980.

\bibitem{MR667967}
S.~G. Krantz.
\newblock Fractional integration on {H}ardy spaces.
\newblock {\em Studia Math.}, 73(2):87--94, 1982.

\bibitem{MR0215072}
R.~Le\'{s}niewicz.
\newblock On {H}ardy-{O}rlicz spaces. {I}.
\newblock {\em Bull. Acad. Polon. Sci. S\'{e}r. Sci. Math. Astronom. Phys.},
  14:145--150, 1966.

\bibitem{MR2927673}
G.~Lu and Y.~Zhu.
\newblock Bounds of singular integrals on weighted {H}ardy spaces and discrete
  {L}ittlewood-{P}aley analysis.
\newblock {\em J. Geom. Anal.}, 22(3):666--684, 2012.

\bibitem{martell-prisuelos}
J.~M. Martell and C.~Prisuelos-Arribas.
\newblock Weighted {H}ardy spaces associated with elliptic operators. {P}art
  {I}: {W}eighted norm inequalities for conical square functions.
\newblock {\em Trans. Amer. Math. Soc.}, 369(6):4193---4233, 2017.

\bibitem{MR2399059}
S.~Meda, P.~Sj\"{o}gren, and M.~Vallarino.
\newblock On the {$H^1$}-{$L^1$} boundedness of operators.
\newblock {\em Proc. Amer. Math. Soc.}, 136(8):2921--2931, 2008.

\bibitem{musielak83}
J.~Musielak.
\newblock {\em Orlicz spaces and modular spaces}, volume 1034 of {\em Lecture
  Notes in Mathematics}.
\newblock Springer-Verlag, Berlin, 1983.

\bibitem{MR2899976}
E.~Nakai and Y.~Sawano.
\newblock Hardy spaces with variable exponents and generalized {C}ampanato
  spaces.
\newblock {\em J. Funct. Anal.}, 262(9):3665--3748, 2012.

\bibitem{MR3233547}
P.~Rocha and M.~Urciuolo.
\newblock Fractional type integral operators on variable {H}ardy spaces.
\newblock {\em Acta Math. Hungar.}, 143(2):502--514, 2014.

\bibitem{MR831197}
F.~J. Ruiz~Blasco and J.~L. Torrea~Hern\'{a}ndez.
\newblock Weighted and vector-valued inequalities for potential operators.
\newblock {\em Trans. Amer. Math. Soc.}, 295(1):213--232, 1986.

\bibitem{MR3090168}
Y.~Sawano.
\newblock Atomic decompositions of {H}ardy spaces with variable exponents and
  its application to bounded linear operators.
\newblock {\em Integral Equations Operator Theory}, 77(1):123--148, 2013.

\bibitem{MR3687096}
Y.~Sawano, K.-P. Ho, D.~Yang, and S.~Yang.
\newblock Hardy spaces for ball quasi-{B}anach function spaces.
\newblock {\em Dissertationes Math.}, 525:102, 2017.

\bibitem{stein93}
E.~M. Stein.
\newblock {\em Harmonic analysis: real-variable methods, orthogonality, and
  oscillatory integrals}, volume~43 of {\em Princeton Mathematical Series}.
\newblock Princeton University Press, Princeton, NJ, 1993.
\newblock With the assistance of Timothy S. Murphy, Monographs in Harmonic
  Analysis, III.

\bibitem{MR1011673}
J.-O. Str{\"o}mberg and A.~Torchinsky.
\newblock {\em Weighted {H}ardy spaces}, volume 1381 of {\em Lecture Notes in
  Mathematics}.
\newblock Springer-Verlag, Berlin, 1989.

\bibitem{MR766221}
J.-O. Str\"{o}mberg and R.~L. Wheeden.
\newblock Fractional integrals on weighted {$H^p$} and {$L^p$} spaces.
\newblock {\em Trans. Amer. Math. Soc.}, 287(1):293--321, 1985.

\bibitem{MR996824}
B.~E. Viviani.
\newblock An atomic decomposition of the predual of {${\rm BMO}(\rho)$}.
\newblock {\em Rev. Mat. Iberoamericana}, 3(3-4):401--425, 1987.

\bibitem{MR3586020}
D.~Yang, Y.~Liang, and L.~D. Ky.
\newblock {\em Real-variable theory of {M}usielak-{O}rlicz {H}ardy spaces},
  volume 2182 of {\em Lecture Notes in Mathematics}.
\newblock Springer, Cham, 2017.

\end{thebibliography}

\end{document}